\documentclass[11pt]{amsart}
\usepackage{amssymb,amscd,amsmath,enumerate}
\usepackage[all]{xy}
\usepackage[utf8]{inputenc}
\usepackage[english]{babel}
\usepackage{tikz}
\usepackage{mathtools}
\usepackage[colorlinks,linkcolor=blue]{hyperref}

\newtheorem{thm}{Theorem}

\newtheorem{cor}[thm]{Corollary}

\newtheorem{prop}[thm]{Proposition}
\newtheorem{lemma}[thm]{Lemma}

\newtheorem{lem}[thm]{Lemma}
\theoremstyle{remark}

\theoremstyle{definition}

\numberwithin{equation}{subsection}

\renewcommand{\bar}{\overline}

\newcommand{\C}{{\mathbb{C}}}

\newcommand{\F}{{\mathbb{F}}}

\newcommand{\Z}{{\mathbb{Z}}}

\newcommand{\Hom}{\mathrm{Hom}}

\newcommand{\GL}{\mathrm{GL}}
\newcommand{\SL}{\mathrm{SL}}

\newcommand{\diag}{\mathrm{diag}}
\newcommand{\Res}{\mathrm{Res}}
\newcommand{\Irr}{\mathrm{Irr}\,}

\newcommand{\id}{\mathrm{id}}

\newcommand{\Alt}{{\raise 2pt\hbox{$\scriptstyle\bigwedge$}}}

\title{On commutator fibers over regular semisimple and central elements of $SL_n$}
\author{Zhipeng Lu}
\address{Mathematisches Institut\\
	Georg-August Universit\"{a}t G\"{o}ttingen\\
	Bunsenstra\ss e 3-5\\
	D-37073 G\"{o}ttingen\\
	Germany}
\email{Zhipeng.Lu@uni-goettingen.de}
\begin{document}
	\maketitle
\section{Introduction}
        In \cite{Larsen-Lu}, we studied the commutator map on special linear groups and proved the equality of dimension of almost all of its fibers. In particular, we showed that the commutator morphism is flat over $\SL_n(\C)$ ($n\geq 2$) away from its center. Let $q$ denote a prime power and define $G_n$ to be $\GL_n(\F_q)$. The crucial step in \cite{Larsen-Lu} is to first establish an analogous result, so called \textit{numerical flatness} on  $\SL_n(\F_q)$, by estimating character values of non-central elements in $G_n$. In more details, let $[,]\colon G_n^2\to \SL_n(\F_q)$ be defined by the commutator $[a,b]=aba^{-1}b^{-1}$, then for any $c\in\SL_n(\F_q)$ non-central, \cite{Larsen-Lu} shows that
        \[|[,]^{-1}(c)|=O(q),\]
        in which $|A|$ denotes the number of elements of the set $A$. Here the implicit constant in the estimate $O(q^m)$ will always be meant to depend on $n$ but not on $q$.
        
        In this paper, we account some detailed calculation of irreducible character values of $G_n$ with application to estimate fiber size of the commutator map over regular semisimple and central elements of $\SL_n$. For regular semisimple elements, we show
        \begin{thm}\label{thm-fiber size over regular semisimples}
        	\[|[,]^{-1}(c)|=(q+O(1))|G_n|,\]
        	for $c$ regular semisimple in $\SL_n(\F_q)$.
        \end{thm}
        In the arithmetic aspect, this improves numerical flatness to \textit{ numerical geometric connectedness} over regular semisimple elements. However the overall geometric connectedness away from center of $\SL_n$ is still a tentative guess, although regular semisimples appear dense.
        
        For central elements of $\SL_n(\F_q)$, we show 
        \begin{thm}\label{thm-over center}
        	If $c\sim \xi I_n\in\SL_n(\F_q)$ with $\xi\in\F_q^\times$ and $\xi^n=1$, then
        	\[|[,]^{-1}(c)=(q^{n/ord(\xi)}+O(1))|G_n|,\]
        	where $ord(\xi)$ is the order of $\xi$ in $\F_q^\times$.
        \end{thm}
        In the geometric aspect we describe components of fibers of the commutator map over center of $\SL_n(\C)$ as follows. In \cite{Larsen-Lu}, geometry of the commutator fibers over primitive $n$-th roots is described by conjugation action of $G_n$. Here we generalize the method to describe fibers over all central elements. See section \ref{section-geometry of fibers} for details.
        
        We outline the proofs here. To prove Theorem \ref{thm-over center}, we first use the following character formula by Frobenius \cite{Frob} to count fiber size of the commutator map:
        
        for any finite group $G$, denote by $\Irr G$ the set of its irreducible characters, then for any $g\in G$,
        \begin{equation}\label{equation-Frobenius formula}
        |\{(x,y)\in G^2\mid [x,y]=g\}|=|G|\sum_{\chi\in\Irr G}\dfrac{\chi(g)}{\chi(1)}.
        \end{equation}
        Then the task becomes to calculate character values of $G_n$, which we establish in section \ref{subsection-character ratio for semisimple elements} and \ref{subsection-character values in general}. We use the results to prove Theorem \ref{thm-fiber size over regular semisimples} in section \ref{subsection-Frobenius on regular semisimple} and Theorem \ref{thm-over center} in section \ref{subsection-Frobenius sum on center}. The basic tool used to evaluate irreducible characters is the generic character formula of $G_n$ established by J. A. Green \cite{Green} which we will give a brief introduction in the next section. At the end, we describe orbits of the conjugation action of $G_n$ on the commutator fibers over center.
        \medskip
        
        \section{Characters of $G_n$}\label{section-characters of Gn}
        \medskip
        
        \subsection{Irreducible characters of $G_n$}\label{subsection-Green's character formula}\hfill\newline
        
        Let $c=(\cdots f^{\nu(f)}\cdots)$ denote a conjugacy class of $G_n$, which assigns a partition $\nu(f)$ to each monic irreducible polynomial $f$ over $\F_q$. For details of Jordan canonical form over $\F_q$, see \cite{Green}. Let $\mathfrak{F}\supset\F_{q}$ be a splitting field of all irreducible polynomials over $\F_q$ of degree not larger than $n$ (e.g. $\F_{q^{n!}}$) and $\eta$ be a primitive element of $\mathfrak{F}^\times$. Let $\theta:\mathfrak{F}^\times\to\C^\times$ be an embedding of multiplication.
        
        For any irreducible polynomial $f$ of degree $s\leq n$ over $\F_q$, if $f(\eta^k)=0$ for some integer $\kappa$, then roots of $f$ are $\eta^{\kappa q^i}, i=0,1,\cdots,s-1$. Such a set $g=\{\kappa,\kappa q,\cdots,\kappa q^{s-1}\}$ is called a $s$-\textit{simplex} by Green \cite{Green}. $s$ is called the \textit{degree} of $g$ and denoted $\deg(g)=s$. The integers in $g$ are all distinct residues modulo $q^{s}-1$ by irreducibility of $f$, called \textit{roots} of $g$. Vice versa, any $s$-simplex determines an irreducible polynomial of degree $s$ over $\F_q$. 
        
        Analogous to conjugacy classes of $G_n$, define a \textit{dual class} of degree $n$ by $e=(g_1^{\nu_1}\cdots g_m^{\nu_m})$ with $g_i$ distinct simplices and $\nu_i$ partitions. If a simplex $g$ appears as a constituent of $e$, we denote it by $g\in e$. The degree of $e$ is defined by the formula 
        \[\sum_{i=1}^{m}\deg(g_i)|\nu_i|,\] 
        and denoted by $\deg(e)=n$. Here $|\lambda|$ is the sum of all parts of the partition $\lambda$. If $m=1$, $e$ is called \textit{primary}. By construction, there are as many dual classes of degree $n$ as conjugacy classes of $G_n$.

       Hereinafter $\rho=\{1^{r_1}2^{r_2}\cdots\}$ denotes the partition having $r_1$ parts equal to $1$, $r_2$ parts equal to $2$ etc, and $N_\rho$ the number of parts in $\rho$. Also for any positive integer $s$, $s\cdot\rho$ denotes the partition $\{s^{r_1}(2s)^{r_2}\cdots\}$. Define a \textit{substitution} into conjugacy classes as a map $\alpha$ sending partitions to rows of irreducible polynomials over $\F_q$, via
       \[\alpha(\rho)=(f_{1,1},\dots f_{1,r_1};f_{2,1},\cdots,f_{2,r_2};\cdots),\]
       with
       \[d(f_{i,j})\mid i, \forall i\geq 1, j\leq r_i,\]
       where $d(f)$ denotes the degree of $f$.
       In other words, it substitutes each part of the partition by an irreducible polynomial of degree dividing the part. For any part $k$ of $\rho$, denoted $k\in\rho$, write $\alpha(k)$ for the polynomial substituting $k$. For any irreducible polynomial $f$ over $\F_q$ and positive integer $k$, let $r_k(\alpha, f)$ be
       the number of parts equal to $k\cdot d(f)$ in $\rho$ which are substituted by $f$ and $\rho(\alpha,f)$ be the partition $\{1^{r_1(\alpha,f)}2^{r_2(\alpha,f)}\cdots\}$. Then  \[d(f)\cdot\rho(\alpha,f):=\{(d(f))^{r_1(\alpha,f)}(2d(f))^{r_2(\alpha,f)}\cdots\}\]
       is the sub-partition of $\rho$ consisting of parts substituted by $f$. We call $\alpha$ a substitution of $\rho$ into $c\sim (f_1^{\nu_1}\cdots f_N^{\nu_N})$ if 
       \[|\rho(\alpha,f_i)|=|\nu_i|, \forall i=1,\cdots,N.\]
       We say two substitutions of $\rho$ are \textit{equivalent} if they are the same under a permutation on equal parts of $\rho$. A collection of all equivalent substitutions is called a \textit{mode} of substitutions. For any two substitutions of the same mode $m$, say $\alpha_1,\alpha_2$ of $\rho$ into $c$, $\rho(\alpha_1, f)=\rho(\alpha_2,f)$ for any irreducible polynomial $f$ occurring in $c$, hence without ambiguity we denote it by $\rho(m,f)$.

       The notion of substitution into dual classes then parallels naturally with parts of partitions substituted by simplices rather than irreducible polynomials. Note that substitution of a partition into a primary dual class is obviously just substituting all parts by the unique simplex.
       
       \begin{prop}[Theorem 14 of \cite{Green}]\label{prop-general character formula}
       Irreducible characters of $\GL_n(\F_q)$ are in one-to-one correspondence with dual classes of degree $n$. For any dual class $e=(g_1^{\lambda_1}\cdots g_k^{\lambda_k})$, denote by $I_e$ the corresponding irreducible character. Then for any conjugacy class $c\sim(f_1^{\nu_1}\cdots f_l^{\nu_l})$ of $\GL_n(\F_q)$ we have \[I_e(c)=(-1)^{n-\sum{|\lambda_i|}}\sum_{\rho,m,m'}\chi(m,e)Q(m',c)B_{\rho}(h^{\rho} m:\xi^\rho m'),\] 
       	where the sum is over all partitions $\rho$ having substitutions $m$ into $e$ and substitutions $m'$ into $c$, 
       	\[\chi(m,e)=\prod_{i=1}^{k}\dfrac{1}{z_{\rho(m,g_i)}}\chi_{\rho(m,g_i)}^{\lambda_i},\]
       	\[Q(m',c)=\prod_{i=1}^{l}\dfrac{1}{z_{\rho(m',f_i)}}Q^{\nu_i}_{\rho(m',f_i)}(q^{d(f_i)}).\]
       	Moreover, the degree of $I_e$ is
       	\begin{align*}
       	&I_e(1)=(-1)^{n-\sum{|\lambda_i|}}\phi_n(q)\prod_{i=1}^{k}\{\lambda_i:q^{\deg(g_i)}\}.
       	\end{align*}
       \end{prop}
       The primary characters have the following more explicit form                
        \begin{prop}[Theorem $12$ of \cite{Green}]\label{prop-primary character formula}
        For any primary character $I_{(g^\lambda)}$ with $\deg(g)=s, s|\lambda|=n$, and any conjugacy class $c\sim (f_1^{\nu_1}\cdots f_l^{\nu_l})$ of $G_n$ with $d(f_i)=d_i, i=1,\cdots,l$, let $\kappa$ be any root of $g$ and $\gamma_i$ any root of $f_i$. Then we have
       \[I_{(g^\lambda)}(c)=(-1)^{n-n/s}\sum_{|\rho|=n/s,m}\chi^\lambda_\rho Q(m,c)\prod_{i=1}^{l}\prod_{\tau\in s\cdot\rho(m,f_i)}T_{s,\tau d_i/s}(\kappa:\gamma_{i}),\] 
       where the sum over all partitions $\rho$ of $n$ and all mode $m$ of substitutions of $s\cdot\rho$ into $c$. Its degree has the formula
       \[I_{(g^\lambda)}(1)=(-1)^{n-n/s}\phi_n(q)\{\lambda:q^s\},\]
       in which $\{\lambda:q\}$ is the Schur function. 
       \end{prop}
       We briefly explain the notations employed in Proposition \ref{prop-general character formula} and Proposition \ref{prop-primary character formula}:
       
       $z_\rho=r_1!2^{r_2}r_2!\cdots k^{r_k}r_k!$ (size of the centralizer of any permutation of $S_n$ having cycle type $\rho$) if $\rho=\{1^{r_1}2^{r_2}\cdots k^{r_k}\}$;
       
       $\chi^\lambda_\rho$ is the character value on permutations of cycle type $\rho$ of the irreducible character of $S_n$ associated to $\lambda$; 
        
       for positive integers $s,d,\kappa$ and any $\gamma\in\mathfrak{F}$ whose minimal polynomial $f$ over $\F_q$ has degree dividing $sd$, i.e. $\gamma\in\F_{q^{sd}}$), define
       \begin{equation}\label{equation-Green's fundamental function}T_{s,d}(\kappa:\gamma)=\theta^{\kappa}(\gamma^{1+q^s+\cdots+q^{(d-1)s}})+\theta^{\kappa q}(\gamma^{1+q^s+\cdots+q^{(d-1)s}})\end{equation} 
       \[+\cdots+\theta^{\kappa q^{s-1}}(\gamma^{1+q^s+\cdots+q^{(d-1)s}}).\] 
       It is easily checked that $T_{s,d}(\kappa:\gamma)=T_{s,d}(\kappa:\gamma^q)$, hence the function is independent of choice of roots of $f$. Therefore we can write $T_{s,d}(\kappa:f)$ without ambiguity;
       
       $\phi_n(q)=(q^n-1)(q^{n-1}-1)\cdots(q-1)$;
       
       $B_{\rho}(h^{\rho}m:\xi^\rho m')$ will be specified later in section \ref{subsection-character values in general};       
       
       $Q(m,c)$ and $\{\lambda:q\}$ will be specified in section \ref{subsection-character ratio for semisimple elements}.

       By the degree formula, we see that linear characters of $G_n$ ($q>2 \text{ or } n>2$) are all of the form $\theta^k(\det(A)), \text{ for } k=1,\cdots, q-1$, where $\det(A)$ is the determinant of $A\in G_n$. Actually, to get $I_e(1)=1$ we can only have $e=(g^\lambda)$ a primary character. Then $|\phi_{\deg(g)|\lambda|}(q)/\phi_{\lambda}(q^{\deg(g)})|=1$ implies $\deg(g)=1$, hence $\lambda=\{n\}$ and the character must be linear.
       \medskip

       \subsection{Frobenius sum}\label{subsection-Frobenius sum}\hfill\newline 
                          
       To start counting $\left|[,]^{-1}(c)\right|$, we rephrase more on Frobenius character formula (\ref{equation-Frobenius formula}) to organize later calculations. We define for any conjugacy class $c$ of $G_n$ the \textit{Frobenius sum}
       \begin{equation}\label{equation-Frobenius sum}\mathcal{S}(c):=\dfrac{\left|[,]^{-1}(c)\right|}{|G_n|}=\sum_{\chi{\in{\Irr G_n}}}\dfrac{\chi(c)}{\chi{(1)}}.\end{equation}
       Using notations of Green's character formula in section \ref{subsection-Green's character formula}, for any array of partitions $\bar{\lambda}=(\lambda_1,\cdots,\lambda_k)$ and $\mathbf{s}=(s_1,\cdots,s_k)\in\Z_+^k$ with $s_1|\lambda_1|+\cdots+s_k|\lambda_k|=n$, any dual class $e=(g_1^{\lambda_1}\cdots g_k^{\lambda_k})$ with $\deg(g_i)=s_i$ possess the same substitutions of partitions. We denote by $(\bar{\lambda},\mathbf{s})$ the set of all such dual classes and call it a \textit{type} and denote by $e\in(\bar{\lambda},\mathbf{s})$ for any dual class $e$ of the type. $I_{\bar{\lambda},\mathbf{s}}(1)$ denotes the degree of irreducible characters associated to dual classes of type $(\bar{\lambda},\mathbf{s})$, which all have the same character degree by Proposition \ref{prop-general character formula}. Then we define the \textit{Frobenius partial sum}
       \begin{align}\label{equation-Frobenius partial sum}
       \mathcal{S}_{\bar{\lambda},\mathbf{s}}(c)=&\sum_{\deg(g_1)=s_1,\cdots,\deg(g_k)=s_k}\dfrac{I_{(g_1^{\lambda_1}\cdots g_k^{\lambda_k})}(c)}{I_{(g_1^{\lambda_1}\cdots g_k^{\lambda_k})}(1)}\\
       =&\dfrac{1}{I_{\bar{\lambda},\mathbf{s}}(1)}\sum_{e\in(\bar{\lambda},\mathbf{s})}I_{(g_1^{\lambda_1}\cdots g_k^{\lambda_k})}(c)\notag
       \end{align}
       Note that by definition of dual classes, $g_i$'s are required to be distinct. 
       
       For $k=1,s=1$ and $\lambda_1=\{n\}$, $g_1$ runs through the $q-1$ residue classes modulo $q-1$ and $I_{(g_1^{\{n\}})}$ are all the linear characters. If $c\in\SL_n(\F_q)$, then $I_{(g_1^{\{n\}})}(c)=I_{(g_1^{\{n\}})}(1)=1$, and
       \[\mathcal{S}_{(\{n\}),(1)}(c)=q-1.\]
       
       Calculating Frobenius partial sums of non-linear types turn out to be difficult for general conjugacy class. In the remaining part of the section, we will try to evaluate them for regular semisimple and central classes.
      \medskip

       \subsection{Primary character values}\label{subsection-character ratio for semisimple elements}\hfill\newline
       
       In this subsection, we aim to establish the following explicit evaluation of primary characters on semi-simple elements in $G_n$
       \begin{prop}\label{prop-semisimple character ratio}
       	Let $e=(g^\lambda)$ with $s=\deg(g)$ and $|\lambda|=n/s$. Suppose $c\sim (f_1^{\{1^{n_1}\}}\cdots f_l^{\{1^{n_l}\}})$ is a semisimple class of $\SL_n(\F_q)$ with $f_i=t-a_i$ for $a_i\in \F_q$ distinct and $\prod_{i=1}^la_i^{n_i}=1$. Then 
       	\[\dfrac{I_e(c)}{I_e(1)}=\theta^{k}(\prod_{i=1}^la_i^{n_i/s})\dfrac{(n/s)!}{(n_1/s)!\cdots (n_k/s)!}\dfrac{\prod_{i=1}^l\phi_{n_i}(q)}{\phi_n(q)}.\]
       	
       \end{prop}
       To prove the proposition, we need a series of lemmas.
       
       \begin{lem}\label{lem-induction of restriction}
       	Let $G$ be any finite group and $V$ be any representation of $G$, then for any subgroup $H$ of $G$ we have
       	\[\mathrm{Ind}_H^G\mathrm{Res}_H^GV\simeq P\otimes V\simeq V\otimes P,\]
       	where $P$ is permutation representation of $G$ on $G/H$.
       \end{lem}
       \begin{proof}
       	Guaranteed by Hall's marriage theorem (see \cite{Hall}), we can choose a set of $g_i$'s in $G$, $i=1,\cdots, m=[G:H]$, such that they are representatives for both left and right cosets of $H$. Then by definition we can construct the induced representation as 
       	\[\mathrm{Ind}_H^G\mathrm{Res}_H^GV=\bigoplus_{i=1}^{m}g_i\otimes V,\]
       	and for any $a\in G$ with $ag_i=g_{\sigma_a(i)}h_i$ (naturally writing $a$ as a permutation of the left cosets of $H$) for some $h_i\in H$ it acts as
       	\[a\cdot(g_1\otimes v_1,\cdots,g_m\otimes v_m)=(g_{\sigma_a(1)}\otimes(h_1\cdot v_1),\cdots,g_{\sigma_a(m)}\otimes(h_m\cdot v_m)),\]
       	in which $v_i\in V$ and $h_i\cdot v_i$ is the restricted action of $H$ on $V$.
       	
       	Now since $g_i$'s are also representatives of the right cosets of $H$, for any $h\in H$, $hg_i=g_ih_i$ for some $h_i\in H$, i.e. $\sigma_h$ is identity as a permutation in $S_m$. Hence $G$ acts on the left component of the tensor product $K[G]\otimes \mathrm{Res}_H^GV\simeq\mathrm{Ind}_H^G\mathrm{Res}_H^GV$ as permutation. 
       \end{proof}
       \begin{lem}[Corollary of Lemma \ref{lem-induction of restriction}]\label{cor-induction of restriction}
       	Denote $\langle\chi_1,\chi_2\rangle|_{H}:=\langle\Res^G_H\chi_1,\Res^G_H\chi_2\rangle$ for any two characters $\chi_1,\chi_2$ of a finite group $G$ and any subgroup $H$ of $G$. Suppose $\chi_1,\chi_2$ are irreducible, then $\langle\chi_1,\chi_2\rangle|_{H}=[G:H]$ if $\chi_1=\chi_2$; otherwise, $\langle\chi_1,\chi_2\rangle|_{H}=0$.
       	
       \end{lem}
       \begin{proof}
       	By Frobenius reciprocity we have
       	\[\langle\Res^G_H\chi_1,\Res^G_H\chi_2\rangle=\langle\chi_1, \mathrm{Ind}_H^G\Res_H^G\chi_2\rangle.\]
       	Let $V_i$ be the representation of $G$ corresponding to $\chi_i, i=1,2$, then by Proposition \ref{lem-induction of restriction} and tensor-hom adjunction we have
       	\[\Hom_G(V_1,\mathrm{Ind}_H^G\Res_H^GV_2)\simeq\Hom_G(V_1,V_2\otimes P)\simeq\Hom_G(V_1,V_2)\otimes P,\]
       	where $P$ is the permutation representation of $G$ on $G/H$. Suppose $V_1,V_2$ are both irreducible, then 
       	\[\Hom_G(V_1,V_2)\simeq\begin{cases}
       	\C, \text{ if } V_1\simeq V_2,\\
       	0, \text{ otherwise}. 
       	\end{cases}\]
       	Hence 
       	\[\Hom_G(V_1,\mathrm{Ind}_H^G\Res_H^GV_2)\simeq\begin{cases}
       	P, \text{ if } V_1\simeq V_2,\\
       	0, \text{ otherwise}. 
       	\end{cases}\]
       	Altogether we get
       	\begin{align*}&\langle\Res^G_H\chi_1,\Res^G_H\chi_2\rangle=\langle\chi_1, \mathrm{Ind}_H^G\Res_H^G\chi_2\rangle\\
       	=&\dim\Hom_G(V_1,\mathrm{Ind}_H^G\Res_H^GV_2)=\begin{cases}
       	[G:H], \text{ if } \chi_1=\chi_2,\\
       	0, \text{ otherwise}. 
       	\end{cases}\end{align*}
       \end{proof}
       
       Schur functions are crucial to our calculation, so we give a brief account of knowledge following D. Littlewood \cite{Littlewood}.
       
       Let $A=(a_{i,j})_{n\times n}$ be any $n$ by $n$ square matrix and $\chi^\lambda$ be the irreducible character of $S_n$ associated to the partition $\lambda$ of $n$. Define the \textit{immanant} of $A$ corresponding to $\lambda$ as 
       \begin{equation}\label{equation-immanant}
       |A|^{(\lambda)}=\sum_{\sigma\in S_n}\chi_\sigma^\lambda a_{1,\sigma(1)}a_{2,\sigma(2)}\cdots a_{n,\sigma(n)},
       \end{equation}  
       where $\chi^\lambda_\sigma$ is the value of $\chi^\lambda$ on $\sigma$.
       For $\lambda=\{1^n\}$, $\chi^{\{n\}}_\sigma=sgn(\sigma)=\pm1$, hence $|A|^{(\{1^n\})}=|A|$ is the ordinary matrix determinant. General immanants are not multiplicative but they satisfy the basic property of conjugate invariance as determinant does, i.e. $|BAB^{-1}|^{(\lambda)}=|A|^{(\lambda)}$ for any invertible $n$ by $n$ matrix $B$.  
       
       For any $n$ variables $\alpha_1,\cdots,\alpha_n$, let $s_i=\alpha_1^i+\cdots+\alpha_n^i$ be the basic symmetric functions on the variables. Define a matrix $s=s(\alpha_1,\cdots,\alpha_n)$ of the following shape
       \[s=\begin{pmatrix}
       s_1&1\\
       s_2&s_1&2\\
       &&&\\
       \cdots&\cdots&\cdots&n-1\\
       s_n&s_{n-1}&\cdots&s_1
       \end{pmatrix},\]
       which comes from the Newton's identities $s(-e_1,e_2,\cdots,(-1)^ne_n)^T=0$ with $e_i$ the elementary symmetric functions on $\alpha_1,\cdots,\alpha_n$. Now for any partition $\lambda$ of $n$ we define the \textit{Schur function} $\{\lambda\}$ on $\alpha_i$'s as
       \begin{equation}\label{equation-Schur function}
       \{\lambda\}=\dfrac{1}{n!}|s|^{(\lambda)}.
       \end{equation}
       We include two basic results on Schur functions introduced in \cite{Littlewood}.
       \begin{lem}[6.2 of \cite{Littlewood}]\label{lem-Schur function identities}
       	For any partition $\lambda=(1^{r_1}2^{r_2}\cdots k^{r_k})$ of $n$, let $s_{\lambda}=s_1^{r_1}s_2^{r_2}\cdots s_k^{r_k}$. Then
       	\[\{\lambda\}=\sum_{|\rho|=n}\dfrac{1}{z_\rho}s_\rho,\]
       	summing over all partitions of $n$. On the other hand, we also have
       	\[s_{\lambda}=\sum_{|\mu|=n}\chi_\lambda^\mu\{\mu\}.\]
       \end{lem} 
       A special case of use is the following
       \begin{lem}\label{lem-rho identity}
       	For any partition $\rho=\{1^{r_1}2^{r_2}\cdots k^{r_k}\}$ of $n$, set \[\beta_\rho(q)=\dfrac{1}{(1-q)^{r_{1}}(1-q^{2})^{r_{2}}\cdots(1-q^k)^{r_k}}.\]
       	Then
       	\[\beta_\rho(q)=\sum_{|\lambda|=n}\chi_\rho^\lambda\{\lambda:q\},\]
       	where $\{\lambda:q\}$ is the Schur function on $1,q,\cdots, q^n$ as of Lemma \ref{lem-Schur function on q}. On the other hand,
       	\[\{\lambda:q\}=\sum_{|\rho|=n}\dfrac{1}{z_\rho}\chi^\lambda_{\rho}\beta_{\rho}(q).\]
       \end{lem}
       \begin{proof}
       	Consider the generalized Schur functions on infinitely many variables introduced in 6.4 of \cite{Littlewood}. Let the variables be $\alpha_i=t^i, i=1,2,\cdots$ for any variable $t$, and formally let $s_r=\sum_{i=1}^\infty\alpha_i^r=1/(1-t^r)$. Then use the similar construction we get identities as of Lemma \ref{lem-Schur function identities}. Finally substitute $t$ by $q$ in the identities which is valid since both sides converge.
       \end{proof}
       Another special case of Schur functions explicitly computable by Jacobi-Trudi equation (see 6.3 of \cite{Littlewood}) is crucial to calculation of character degrees of $G_n$.
       \begin{lem}[7.1 of \cite{Littlewood}]\label{lem-Schur function on q}
       	For any partition $\lambda=\{\lambda_1,\lambda_2,\cdots,\lambda_l\}$ of $n$ with $\lambda_1\geq \lambda_2\geq\dots\geq \lambda_l$, let $n_\lambda=\lambda_2+2\lambda_3+\cdots+(l-1)\lambda_l$ and $\{\lambda:q\}$ be the Schur function on $1,q,\cdots,q^n$ corresponding to $\lambda$, then 
       	\[\{\lambda:q\}=q^{n_\lambda}\dfrac{\prod_{1\leq r<s\leq l}(1-q^{\lambda_r-\lambda_s-r+s})}{\prod_{r=1}^{l}\phi_{\lambda_r+l-r}(q)},\]
       	where $\phi_k(t)=(1-t)\cdots(1-t^k)$ for any positive integer $k$.
       \end{lem}
       
       Recall the definition of simplex from subsection \ref{subsection-Green's character formula}. The roots of simplices of degree $s$ have representatives $k$ between $1$ and $q^{s}-1$ such that $k, kq, \cdots, kq^{s-1}$ are all distinct modulo $q^{s}-1$. We denote by $\mathfrak{K}_s$ the set of such numbers in $\{1,\cdots,q^s-1\}$. 
       
       For example, $\mathfrak{K}_1=\{1,2,\cdots,q-1\}$ by definition. For $s=2$, we need to confine $k\not\equiv kq\ (mod\ q^2-1)$, i.e. $q^2-1\nmid t(q-1)$ or $q+1\nmid t$, hence \[\mathfrak{K}_2=\{1,2,\cdots,q^2-1\}\smallsetminus\{q+1,2(q+1),\cdots,(q-1)(q+1)\}.\] 
       
       In general, if $kq^{i}\equiv kq^j \mod (q^s-1)$ for some $1\leq i< j\leq s-1$, then $(q^s-1)\mid sq^i(q^{j-i}-1)$, or $(q^s-1)|i(q^{j-i}-1)$. 
       Thus we have
       \[\{1\leq t\leq q^s-1\}\smallsetminus\mathfrak{K}_s=\{1\leq t\leq q^s-1,t=l\dfrac{q^s-1}{q^{r}-1},\exists r|s,1\leq r<s\}\]
       \[=\bigcup_{1\leq r<s,r|s}\dfrac{q^s-1}{q^{r}-1}\{1\leq t\leq q^r-1\},\]
       and for any $r|s,r'|s$ we have
       \[\dfrac{q^s-1}{q^{r}-1}\{1\leq k\leq q^r-1\}\cap\dfrac{q^s-1}{q^{r'}-1}\{1\leq k\leq q^{r'}-1\}\]
       \[=\dfrac{q^s-1}{q^{(r,r')}-1}\{1\leq k\leq q^{(r,r')}-1\}.\]
       
       By the inclusion-exclusion principle we immediately get 
       \begin{lem}\label{lem-sum over primary roots}
       	For any $\xi\in\mathbb{F}_{q^s}$ and  $\theta:\mathfrak{F}^\times\rightarrow\C^\times$ fixed in section \ref{subsection-Green's character formula},
       	\[\sum_{\kappa\in\mathfrak{K}_s}\theta^\kappa(\xi)=\sum_{r\mid s}\mu(s/r)\sum_{t=1}^{q^r-1}\theta^{t(q^s-1)/(q^{r}-1)}(\xi)=\sum_{r|\id_s(\xi)}\mu(s/r)(q^r-1),\]
       	in which $\mu$ is the M\"{o}bius function and $\id_s(\xi)$ denotes the largest divisor $r$ of $s$ such that $\xi^{(q^s-1)/(q^r-1)}=1$. 
       \end{lem}
       We call $\id_s(\xi)$ the \textit{identity} of $\xi$ with respect to $s$. Clearly $\id_s(1)=s$. Also, if there is no divisor $r\mid s$ such that $\xi^{(q^s-1)/(q^r-1)}=1$, we let $\id_s(\xi)=0$ so that $\sum_{k\in\mathfrak{K}_s}\theta^k(\xi)=0$ vacantly. 
       
       In addition, we need explicit value of Green's polynomials in some special cases.
       \begin{lem}[Appendix tables of $Q^\lambda_\rho(q)$, \cite{Green}]\label{lem-Green's appendix}
       	For any partitions $\lambda, \rho$ of $n$ with $\rho=\{1^{r_1}2^{r_2}\cdots k^{r_k}\}$, we have
       	\begin{align*}
       	&Q^{\{n\}}_\rho(q)=1,\\
       	&Q^{\{1^n\}}_\rho(q)=\dfrac{\phi_n(q)}{(1-q)^{r_1}(1-q^2)^{r_2}\cdots(1-q^k)^{r_k}},\\
       	&Q^\lambda_{\{n\}}(q)=\phi_{N_\lambda-1}(q),\\
       	&Q^\lambda_\rho(q) \text{ has leading term }\chi^{\lambda}_\rho q^{n_\lambda},
       	\end{align*} 
       	where $n_\lambda=\lambda_2+2\lambda_3+\cdots+(l-1)\lambda_l$ if $\lambda=\{\lambda_1,\lambda_2,\cdots,\lambda_l\}$ with $\lambda_i$ descending. 		
       \end{lem}  
       We define 
       \[\beta_\rho(q):=\dfrac{1}{(1-q)^{r_1}(1-q^2)^{r_2}\cdots(1-q^k)^{r_k}}.\]
       Now we are ready to prove Proposition \ref{prop-semisimple character ratio}.
       \begin{proof}[Proof of Proposition \ref{prop-semisimple character ratio}]
       	By Green's character formula as of Proposition \ref{prop-primary character formula}, for any root $k\in g$, 
       	\begin{equation}\label{proof of thm 1-character formula}I_e(c)=(-1)^{n-n/s}\sum_{|\rho|=n/s,m}\chi^\lambda_\rho Q(m,c)\prod_{i=1}^{l}\prod_{\tau\in\rho(m,f_i)}T_{s,\tau}(k:a_{i}),\end{equation}
       	summing over all modes $m$ of substitutions of $s\cdot\rho$ into $c$. Explicitly, 
       	\[Q(m,c)=\prod_{i=1}^{l}\dfrac{1}{z_{s\cdot\rho(m,f_i)}}Q^{\{1^{n_i}\}}_{s\cdot\rho(m,f_i)}(q)\]
       	By Lemma \ref{lem-Green's appendix}, if $\rho(m,f_i)=\{1^{r_{i,1}}2^{r_{i,2}}\cdots\}$, then 
       	\[Q^{\{1^{n_i}\}}_{s\cdot\rho(M,f_i)}(q)=\phi_{n_i}(q)\dfrac{1}{(1-q^s)^{r_{i,1}}(1-q^{2s})^{r_{i,2}}\cdots}=\phi_{n_i}(q)\beta_{\rho(m,f_i)}(q^s).\]
       	Since $\beta_\nu(t)$ is clearly multiplicative, i.e. $\beta_{\nu}(t)\cdot\beta_{\mu}(t)=\beta_{\nu+\mu}(t)$ for $\nu+\mu$ the partition collecting all parts of $\nu$ and $\mu$, we can write
       	\[\prod_{i=1}^{l}Q^{\{1^{n_i}\}}_{s\cdot\rho(m,f_i)}(q)=\phi_{n_1}(q)\cdots\phi_{n_l}(q)\beta_\rho(q^s).\]       	
       	As to the function $T$, by the expression (\ref{equation-Green's fundamental function}),
       	\begin{align*}\prod_{i=1}^{l}\prod_{\tau\in\rho(m,f_i)}T_{s,\tau}(k:a_{i})=&\prod_{i=1}^l\prod_{\tau\in\rho(m,f_i)}\sum_{j=0}^{s-1}\theta^{kq^{j}}(a_i^{\tau})\\
       	=&\prod_{i=1}^l\prod_{\tau\in\rho(M,f_i)}s\theta^{k}(a_i^\tau)\\
       	=&s^{N_\rho}\theta^{k}(\prod_{i=1}^la_i^{n_i/s}).
       	\end{align*}
       	Now plug in these acquisitions to (\ref{proof of thm 1-character formula}) we get 
       	\begin{align*} &(-1)^{n-n/s}I_e(c)\\
       	=&\theta^{k}\left(\prod_{i=1}^la_i^{n_i/s}\right)\prod_{i=1}^l\phi_{n_i}(q)\sum_{|\rho|=n/s,m}\chi^\lambda_\pi s^{N_\rho}\prod_{i=1}^{l}\dfrac{1}{z_{s\cdot\rho(m,f_i)}}\beta_{\rho(m,f_i)}(q^s)\\
       	=&\theta^{k}\left(\prod_{i=1}^la_i^{n_i/s}\right)\prod_{i=1}^l\phi_{n_i}(q)\sum_{|\rho_i|=n_i/s}\left(\prod_{i=1}^{l}\dfrac{1}{z_{\rho_i}}\right)\chi^\lambda_{\rho_1+\cdots+\rho_l}\beta_{\rho_1+\cdots+\rho_l}(q^s).\end{align*}
       	By Lemma \ref{lem-rho identity}, we can write
       	\[\beta_{\rho_1+\cdots+\rho_l}(q^s)=\sum_{|\mu|=n/s}\chi_{\rho_1+\cdots+\rho_l}^{\mu}\{\mu:q^s\},\]
       	hence
       	\begin{align*}
       	&\sum_{|\rho_i|=n_i/s}\left(\prod_{i=1}^{l}\dfrac{1}{z_{\rho_i}}\right)\chi^\lambda_{\rho_1+\cdots+\rho_l}\beta_{\rho_1+\cdots+\rho_l}(q^s)\\
       	=&\sum_{|\rho_i|=n_i/s}\left(\prod_{i=1}^{l}\dfrac{1}{z_{\rho_i}}\right)\chi^\lambda_{\rho_1+\cdots+\rho_l}\sum_{|\mu|=n/s}\chi_{\rho_1+\cdots+\rho_l}^{\mu}\{\mu:q^s\}\\
       	=&\sum_{|\mu|=n/s}\{\mu:q^s\}\sum_{|\rho_i|=n_i/s}\left(\prod_{i=1}^{l}\dfrac{1}{z_{\rho_i}}\right)\chi^\lambda_{\rho_1+\cdots+\rho_l}\chi_{\rho_1+\cdots+\rho_l}^{\mu}\\
       	=&\sum_{|\mu|=n/s}\{\mu:q^s\}\langle\chi^\lambda,\chi^\mu\rangle|_{S_{n_1/s}\times\cdots\times S_{n_k/s}}.
       	\end{align*}
       	By Lemma \ref{cor-induction of restriction} we have
       	\[\langle\chi^\lambda,\chi^\mu\rangle|_{S_{n_1/s}\times\cdots\times S_{n_k/s}}=[S_{n/s}:S_{n_1/s}\times\cdots\times S_{n_k/s}]=\dfrac{(n/s)!}{(n_1/s)!\cdots (n_k/s)!}\]
       	if $\mu=\lambda$, or $0$ otherwise. Hence
       	\[\sum_{|\pi_i|=n_i/s}\left(\prod_{i=1}^{l}\dfrac{1}{z_{\pi_i}}\right)\chi^\lambda_{\pi_1+\cdots+\pi_l}\beta_{\pi_1+\cdots+\pi_l}(q^s)=\dfrac{(n/s)!}{(n_1/s)!\cdots (n_k/s)!}\{\lambda:q^s\}.\] 
       	and this in turn gives
       	\[I_e(c)=(-1)^{n-n/s}\theta^{k}(\prod_{i=1}^la_i^{n_i/s})\dfrac{(n/s)!}{(n_1/s)!\cdots (n_k/s)!}\left(\prod_{i=1}^l\phi_{n_i}(q)\right)\{\lambda:q^s\}.\]
       	By the character degree formula of Proposition \ref{prop-primary character formula} that $I_e(1)=(-1)^{n-n/s}\phi_n(q)\{\lambda:q^s\}$,
       	we finally get the character ratio
       	\[\dfrac{I_e(c)}{I_e(1)}=\theta^{k}(\prod_{i=1}^la_i^{n_i/s})\dfrac{(n/s)!}{(n_1/s)!\cdots (n_k/s)!}\dfrac{\prod_{i=1}^l\phi_{n_i}(q)}{\phi_n(q)}.\]
       \end{proof}
       \begin{cor}
       	With the same notations as of Proposition \ref{prop-semisimple character ratio},
       	\[\mathcal{S}_{(\lambda),(s)}(c)=\dfrac{(n/s)!}{(n_1/s)!\cdots (n_k/s)!}\dfrac{\prod_{i=1}^l\phi_{n_i}(q)}{\phi_n(q)}\dfrac{1}{s}\sum_{r\mid \id_s(\prod_{i=1}^la_i^{n_i/s})}\mu(s/r)(q^r-1),\]
       	where $\id_s$ is as of Lemma \ref{lem-sum over primary roots} and $\mu$ is the M\"obius function.
       \end{cor}        
       
       For any $a\in\F_q^\times$, integer $s$ and its divisor $r$, 
       \[a^{\frac{q^s-1}{q^r-1}}=a^{\frac{s}{r}}.\]
       If its order $ord(a)$ divides $s$, then we see that $\id_s(a)=\frac{s}{ord(a)}$; otherwise $id_s(a)=0$.             
       Note that for $c$ as in Proposition \ref{prop-semisimple character ratio}, $\prod_{i=1}^la_i^{n_i}=1$ and $a_i\in\F_q$, the order $d$ of $\prod_{i=1}^la_i^{n_i/s}$ in $\F_q^\times$ divides $s$, hence $\id_s(\prod_{i=1}^la_i^{n_i/s})=\dfrac{s}{d}$. 
       
       By obvious estimate, 
       \[\dfrac{\prod_{i=1}^l\phi_{n_i}(q)}{\phi_n(q)}=(1+o(1))q^{-\sum_{1\leq i<j\leq l}n_in_j}\]
       and
       \[\sum_{r\mid \id_s(\prod_{i=1}^la_i^{n_i/s})}\mu(s/r)(q^r-1)=(1+o(1)) q^{s/d},\]
       or $o(q^{s/d})$, depending on $\mu(d)$. Since $s\mid n_i$, if $l\geq 2$, i.e. $c$ is non-central semisimple, then $-\sum_{1\leq i<j\leq l}n_in_j+s/d\leq 0$. Thus we prove
       \begin{cor}\label{cor-character ratio for non-scalar semisimple elements}
       	Suppose $c$ is a non-scalar semisimple conjugacy class and $e=(g^\lambda)$ as in Proposition \ref{prop-semisimple character ratio}. Let $d$ be the order of $\prod_{i=1}^la_i^{n_i/s}$ in $\F_q^\times$. If $d$ divides $s$, then
       	\[\mathcal{S}_{(\lambda),(s)}(c)=(\dfrac{1}{s}+o(1))q^{-\sum_{1\leq i<j\leq l}n_in_j+s/d}.\]
       	If $d\nmid s$, then $\mathcal{S}_{(\lambda),(s)}(c)=0$. 
       \end{cor} 
       In either case, we prove 
       \begin{cor}\label{cor-partial Frobenius sum on semisimple}
       	With notations as above, the Frobenius partial sums of primary types on non-central semisimple $c$ have
       	\[\mathcal{S}_{(\lambda),(s)}(c)=O(1).\]          	
       \end{cor}
       This result relates to our tentative guess that $\mathcal{S}_{\bar{\lambda},\mathbf{s}}(c)=O(1)$ for general type on non-central conjugacy class, which is supposed to improve the numerical flatness result of \cite{Larsen-Lu} to \textit{numerical geometric connectedness}.
       \medskip
               
       \subsection{General character values}\label{subsection-character values in general}\hfill\newline
       
       Set $e=\left(g_1^{\lambda_1}\cdots g_k^{\lambda_k}\right)$ with $s_i=\deg(g_i)$ all distinct and $c\sim \left(f_1^{\{1^{n_1}\}}\cdots f_l^{\{1^{n_l}\}}\right)$ is a semisimple class of $\SL_n(\F_q)$ with $f_i=t-a_i$ for $a_i\in \F_q$ distinct and $\prod_{i=1}^la_i^{n_i}=1$. We will specify the term $B_\rho(h^\rho m:\xi^\rho m')$ in Green's general character as of Proposition \ref{prop-general character formula} in this case. 
       
       For any $\rho=(1^{r_1}2^{r_2}\cdots)$, associate a row of variables \[Y^\rho=\{y_{11},\cdots,y_{1r_1};y_{21},\cdots,y_{2r_2};\cdots\}.\] Then a substitution $m$ of $\rho$ into dual classes (conjugacy classes) can be seen as a substitution of variables $y_{di}$ by simplices of degree dividing $d$ (monic irreducible polynomials of degree dividing $d$ over $\F_q$) which we denote by $y_{di}m$. Theorem 14 of \cite{Green} expresses $B_\rho(h^\rho:\xi^\rho)$ with intermediates \[h^\rho:=\{h_{11},\cdots,h_{1r_1};h_{21},\cdots,h_{2r_2};\cdots\},\]
       and \[\xi^\rho:=\{\xi_{11},\cdots,\xi_{1r_1};\xi_{21},\cdots,\xi_{2r_2};\cdots\},\]
       as follows:
       \[B_\rho(h^\rho:\xi^\rho)=\prod_{d}\sum_{1',\cdots, r_d'}S_d(h_{d1}:\xi_{d1'})\cdots S_{d}(h_{dr_d}:\xi_{dr_d'})\]
       where the inner sum is over all permutations $1',\cdots, r_d'$ of $1,\cdots,r_d$ and
       \[S_d(h:\xi)=\theta^h(\xi)+\theta^{hq}(\xi)+\cdots+\theta^{hq^{d-1}}(\xi)\]
       with the embedding $\theta:\mathfrak{F}\to\C^\times$ fixed in the beginning of section \ref{subsection-Green's character formula}. 
       
       For any substitution $m$ of $\rho$ into $e$, by Definition 8.1 of \cite{Green}, the intermediate $h_{di}$ is substituted by
       \[h_{di}m=\kappa_{di}\dfrac{q^d-1}{q^{s_{di}}-1},\]
       where $\kappa_{di}\in y_{di}m$ and $s_{di}=\deg(y_{di}m)$. For any substitution $m'$ of $\rho$ into $c$, by Definition 4.9 of \cite{Green}, the intermediate $\xi_{di}$ is substituted by $\xi_{di}m'=\gamma_{di}$ for any root $\gamma_{di}$ of $y_{di}m'$. Clearly we can write
       \[S_d(h_{di}m:\xi_{di'}m')=\dfrac{d}{s_{di}}\sum_{\kappa\in h_{di}m}\theta^{\kappa\frac{q^d-1}{q^{s_{di}}-1}}(\xi_{di'}m').\] 
       Let $\mathcal{S}_\rho$ be the permutation group permuting parts of $\rho$ and preserving its equal parts. Let $\sigma(\tau)$ be the part that $\tau$ is sent to for each $\sigma\in\mathcal{S}_\rho$ and $\tau\in\rho$. Then recalling the notion of substitution from section \ref{subsection-Green's character formula} we can write
       \begin{equation}\label{equation-B rho explicit in general}
       B_\rho(h^\rho m:\xi^\rho m')=\sum_{\sigma\in\mathcal{S}_\rho}\prod_{g\in e}\prod_{\tau\in\rho(m,g)}\tau\sum_{\kappa\in g}\theta^{\kappa\frac{q^{\tau\deg(g)}-1}{q^{\deg(g)}-1}}(\xi_{\sigma(\tau)}m').
       \end{equation}
       
       For $c$ semisimple,  $\xi_{di}m'$ all lie in $\F_q$. Hence 
       \[S_d(h_{di}m:\xi_{di'}m')=d\theta^{h_{di}m}(\xi_{di'}m')=d\theta^{\kappa_{di}d/s_{di}}(\xi_{di'}m'),\]
       and
       \begin{equation}\label{equation-B for semisimple}
       B_\rho(h^\rho m:\xi^\rho m')
       =\sum_{\sigma\in\mathcal{S}_\rho}\prod_{g\in e}\prod_{\tau\in\rho(m,g)}\tau\deg(g)\theta^{\kappa_g\tau}(\xi_{\sigma(\tau)}m'),
       \end{equation}
       where $\kappa_g\in g$. In general, the sums in (\ref{equation-B rho explicit in general}) and (\ref{equation-B for semisimple}) are impenetrable by simple arrangement. Rather we resort to the Frobenius partial sum on dual classes of a given type. First by Lemma \ref{lem-Green's appendix} we get
       \begin{align*}Q(m',c)=&\prod_{i=1}^l\dfrac{1}{z_{\rho(m',f_i)}}Q^{\{1^{n_i}\}}_{\rho(m',f_i)}(q^{d(f_i)})\\
       =&\phi_{n_1}(q)\cdots\phi_{n_l}(q)\prod_{i=1}^l\dfrac{1}{z_{\rho(m',f_i)}}\beta_{\rho(m',f_i)}(q)\\
       =&\phi_{n_1}(q)\cdots\phi_{n_l}(q)\dfrac{\beta_\rho(q)}{\prod_{i=1}^lz_{\rho(m',f_i)}},\end{align*}
       since $\beta_\rho$ is additive and $\rho(m',f_1)+\cdots+\rho(m',f_l)=\rho$ in our case. Then by Proposition \ref{prop-general character formula} and (\ref{equation-Frobenius partial sum}) we get for $\bar{\lambda}=(\lambda_1,\cdots,\lambda_k)$ and $\mathbf{s}=(s_1,\cdots,s_k)$
       \begin{align}\mathcal{S}_{\bar{\lambda},\mathbf{s}}(c)=&\dfrac{1}{I_{\bar{\lambda},\mathbf{s}}(1)}\sum_{\deg(g_1)=s_1,\cdots,\deg(g_k)=s_k}I_{\left(g_1^{\lambda_1}\cdots g_k^{\lambda_k}\right)}(c)\notag\\
       =&\dfrac{(-1)^{n-\sum_{j=1}^k|\lambda_j|}}{I_{\bar{\lambda},\mathbf{s}}(1)}\sum_{e\in(\bar{\lambda},\mathbf{s})}\sum_{\rho,m,m'}\chi(m,e)Q(m',c)B_\rho(h^\rho m:\xi^\rho m')\notag\\
       \label{equation-general Frobenius partial sum}=&\dfrac{(-1)^{n-\sum_{j=1}^k|\lambda_j|}}{I_{\bar{\lambda},\mathbf{s}}(1)}\sum_{\rho,m,m'}\chi(m,e)Q(m',c)\sum_{e\in(\bar{\lambda},\mathbf{s})}B_\rho(h^\rho m:\xi^\rho m')\\
       \label{equation-semisimple Frobenius partial sum}=&\dfrac{(-1)^{n-\sum_{j=1}^k|\lambda_j|}\prod_{i=1}^l\phi_{n_i}(q)}{I_{\bar{\lambda},\mathbf{s}}(1)}\sum_{\rho,m,m'}\dfrac{\beta_\rho(q)\chi(m,e)}{\prod_{i=1}^lz_{\rho(m',f_i)}}\sum_{e\in(\bar{\lambda},\mathbf{s})}B_\rho(h^\rho m:\xi^\rho m'),
       \end{align}
       where (\ref{equation-general Frobenius partial sum}) is for general conjugacy class and (\ref{equation-semisimple Frobenius partial sum}) is for semisimple class $c\sim \left(f_1^{\{1^{n_1}\}}\cdots f_l^{\{1^{n_l}\}}\right)$. Note that substitutions can be defined unambiguously from $\rho$ into the type $(\bar{\lambda},\mathbf{s})$ since they are commonly shared by each dual class in it. The inner sum
       \begin{equation}\label{equation-B in general}
       B_{\rho,\bar{\lambda},\mathbf{s},c,m,m'}:=\sum_{e\in(\bar{\lambda},\mathbf{s})}B_\rho(h^\rho m:\xi^\rho m')
       \end{equation}
       is key to our calculation. 
       By (\ref{equation-B rho explicit in general}) we have
       \begin{equation}\label{equation-B explicity}B_{\rho,\bar{\lambda},\mathbf{s},c,m,m'}=\sum_{\sigma\in\mathcal{S}_\rho}\prod_{g\in e}\sum_{\kappa\in\mathfrak{K}_{\deg(g)}}^*\prod_{\tau\in\rho(m,g)}\tau\theta^{\kappa \frac{q^{\tau\deg(g)}-1}{q^{\deg(g)}-1}}\left(\xi_{\sigma(\tau)}m'\right),\end{equation}
       in which $\sum\limits^*$ indicates avoiding duplicates of simplices if $e$ contains multiple $s$-simplices. If simplices appearing in $e$ all have distinct degrees, we call its type $(\bar{\lambda},\mathbf{s})$ a \textit{distinct} type. In this case $\sum\limits^*$ becomes $\sum$ and Lemma \ref{lem-sum over primary roots} immediately gives
       \begin{prop}\label{prop-B explicit in general}
       	With the notations above, for general conjugacy class $c$ and dual class $e=(g_1^{\lambda_1}\cdots g_k^{\lambda_k})$ belonging to the distinct type $(\bar{\lambda},\mathbf{s})$, we have
       	\[B_{\rho,\bar{\lambda},\mathbf{s},c,m,m'}=\sum_{\sigma\in\mathcal{S}_\rho}\prod_{i=1}^k\prod_{\tau\in\rho(m,g_i)}\tau\sum_{r\mid\id_{s_i}\left(\prod_{\tau\in\rho(m,g_i)}\left(\xi_{\sigma(\tau)}m'\right)^{\frac{q^{\tau s_i}-1}{q^{s_i}-1}}\right)}\mu\left(\frac{s_i}{r}\right)(q^r-1).\]
       \end{prop} 
       This result will be mainly used to deal with general irreducible character values on central and regular semisimple elements.
   \medskip
   
       \subsection{Frobenius sum on regular semisimple elements}\label{subsection-Frobenius on regular semisimple}\hfill\newline
       
       Provided the previous results, the Frobenius sum on regular semisimple elements turn out easiest to be estimated. We first introduce the following lemmas as preparation.
       \begin{lem}\label{lem-degree of simple characters} Suppose $\deg(g)=s$, then
       	\[I_{(g^{\{n/s\}})}(1)=\dfrac{\phi_n(q)}{\phi_{n/s}(q^s)}=(1+o(1))q^{n(n-n/s)/2}.\]
       \end{lem}
       \begin{proof}
       With notations of Lemma \ref{lem-Schur function on q}, for $\lambda=\{n/s\}$, $n_\lambda=0$, hence
       \[\{\lambda:q^s\}=\dfrac{1}{\phi_{n/s}(q^s)}.\] 
       The lemma follows from the character degree formula of Proposition \ref{prop-primary character formula}.
       \end{proof}
       \begin{lem}\label{lem-degree comparison lemma}
       	$I_{(g^\lambda)}(1)\geq I_{(g^{\{n/\deg(g)\}})}(1)$ in terms of highest power of $q$, for any $|\lambda|=v$. In general, for $e=(g_1^{\lambda_1}\cdots g_k^{\lambda_k})$ with $\deg(g_i)=s_i$, $|\lambda_i|=\nu_i$ and $\sum_{i=1}^ks_i\nu_i=n$, we have $I_{e}(1)\geq I_{(g_1^{\{\nu_1\}}\cdots g_k^{\{\nu_k\}})}(1)$. More explicitly,
       	\[I_{e}(1)\geq \dfrac{\phi_n(q)}{\prod_{i=1}^k\phi_{\nu_i}(q^{s_i})}=(1+o(1))q^{(n^2-\sum_{i=1}^ks_i\nu_i^2)/2}.\]
       \end{lem} 
       \begin{proof}
       	By Proposition \ref{prop-primary character formula} we have
       	\[I_{(g^\lambda)}(1)=\phi_n(q)\{\lambda:q^s\},\]
       	and by Lemma \ref{lem-Schur function on q}
       	\[\{\lambda:q\}=\dfrac{q^{l_2+2l_3+\cdots}\prod_{1\leq r<t\leq u}(1-q^{l_r-l_t-r+t})}{\prod_{r=1}^{u}\phi_{l_r+u-r}(q)},\]
       	if $\lambda=\{l_1,l_2,\cdots,l_u\}$ with $l_1\geq l_2\geq\cdots\geq l_u>0$.
       	To prove the desired result, we just need to show $\{\lambda:q\}\gtrsim \phi^{-1}_{|\lambda|}(q)$. Rather than directly compare them, we comprise to show
       	\[\{\lambda:q\}\phi_{|\lambda|}(q)\gtrsim\{\lambda:q\}\prod_{r=1}^{u}\phi_{l_r}(q)\gtrsim 1.\]
       	By the above formula we get
       	\[\{\lambda:q\}\prod_{r=1}^{u}\phi_{l_r}(q)\sim q^{l_2+2l_3+\cdots+\sum_{1\leq r<t\leq u}(l_r-l_t-r+t)-\sum_{r=1}^{u}\sum_{t=1}^{u-r}(l_r+t)}.\]
       	We compute the middle sum that
       	\[\sum_{1\leq r<t\leq u}(l_r-l_t-r+t)=\sum_{r=1}^{u}(u-r)l_r-\sum_{r=1}^{u}(r-1)l_r+\sum_{r=1}^{u}(u-2r+1)(u-r),\]
       	then the starting sum that
       	\[l_2+2l_3+\cdots=\sum_{r=1}^{u}(r-1)l_r,\]
       	and the last sum that
       	\[\sum_{r=1}^{u}\sum_{t=1}^{u-r}(l_r+t)=\sum_{r=1}^{u}[(u-r)l_r+(u-r+1)(u-r)/2].\]
       	Then combining all those we get
       	\[\sum_{r=1}^{u}(u-r)(u-3r+1)/2.\]
       	Then this sum is supposed to vanish and it is indeed so for $u\geq 1$. Thus we proved the desired result in the primary case.
       	
       	In general for $e=(g_1^{\lambda_1}\cdots g_k^{\lambda_k})$, by Proposition \ref{prop-general character formula} we have
       	\[I_e(1)=\dfrac{\phi_n(q)}{\prod_{i=1}^k\phi_{s_i\nu_i}(q)}\prod_{i=1}^kI_{(g_i^{\lambda_i})}(1).\]
       	Then the inequality follows from that of the primary case.
       \end{proof}
       \begin{prop}[Theorem \ref{thm-fiber size over regular semisimples}]\label{prop-regular semisimple}
       	For any regular semisimple $c\in\SL_n(\F_q)$,  
       	\[\sum\dfrac{\chi(c)}{\chi(1)}\sim O(1),\]
       	where the sum is over all non-linear irreducible characters of $G_n$.
       \end{prop}
       \begin{proof}
       	Assume $c\sim\mathrm{diag}(a_1,\cdots, a_n)$ with $a_i$ all distinct and $a_1\cdots a_n=1$. It suffices need to show each Frobenius partial sum over any non-linear type is $O(1)$. The primary case has been verified by Corollary \ref{cor-partial Frobenius sum on semisimple}, so we just need to consider for the type $(\bar{\lambda},\mathbf{s})$ of $e=(g_1^{\lambda_1}\cdots g_k^{\lambda_k})$ with $k\geq 2$. By Proposition \ref{prop-general character formula}, $I_e(c)\neq 0$ only for dual classes $e=(g_1^{\lambda_1}\cdots g_k^{\lambda_k})$ with $\deg(g_i)=1$, and the only partition having substitution into $c$ is $\{1^n\}$. Then
       	\[Q(m',c)=1,\chi(m,e)=\prod_{j=1}^k\dfrac{1}{|\lambda_j|!}\chi^{\lambda_j}(1),\]
       	since $\rho(m,g_j)=\{1^{|\lambda_j|}\}$, and (\ref{equation-B rho explicit in general}) becomes
       	\begin{align}\label{equation-B rho for regular semisimple}B_{\rho}(h^{\rho} m:\xi^{\rho} m')=&\sum_{\sigma\in S_n}\prod_{j=1}^k\prod_{\tau\in\{1^{|\lambda_j|}\}}\theta^{\kappa_j}(\xi_{\sigma(\tau)}m')\end{align}
       	Hence (\ref{equation-general Frobenius partial sum}) becomes
       	\[\mathcal{S}_{\bar{\lambda},\mathbf{s}}(c)=\dfrac{B_{\{1^n\},\bar{\lambda},\mathbf{s},c.m,m'}}{I_{e}(1)\prod_{j=1}^k|\lambda_j|!}.\]
       	Clearly $B_{\{1^n\},\bar{\lambda},\mathbf{s},c.m,m'}=O(q^k)$. By Lemma \ref{lem-degree comparison lemma}, for $k\geq 2$ we have in terms of power of $q$,
       	\[I_{e}(1)\geq q^{(n^2-\sum_{i=1}^ks_i\nu_i^2)/2}\geq q^{(n^2-\sum_{i=1}^k(s_i\nu_i)^2)/2}.\] 
       	Let $n_i=s_i\nu_i$, then $\sum_{i=1}^kn_i=n$ and 
       	\[n^2-(n_1^2+\cdots+n_k^2)=\sum_{1\leq i\neq j\leq k}n_in_j\geq k^2-k.\]
       	Thus in terms of power of $q$,
       	\[\mathcal{S}_{\bar{\lambda},\mathbf{s}}(c)=\dfrac{B_{\{1^n\},\bar{\lambda},\mathbf{s},c.m,m'}}{I_{e}(1)}\leq q^{k-(k^2-k)}=q^{2k-k^2}=O(1)\]
       	for $k\geq 2$. Hence the proof of the proposition or Theorem \ref{thm-fiber size over regular semisimples}.
       \end{proof}
       \medskip

        \subsection{Frobenius sum on center of $\SL_n(\F_q)$}\label{subsection-Frobenius sum on center}\hfill\newline
          
        If $c$ is a scalar matrix of $\SL_n(\F_q)$, say $c=\xi I_n$ with $\xi\in\F_q$, denoted by $\xi$ without ambiguity hereinafter, we have $\xi^n=1$, i.e. $ord(\xi)\mid n$. Moreover if $ord(\xi)$ divides $s$ (note that $s\mid n$) then $\id_s(\xi)=\frac{s}{ord(\xi)}$. Thus from Proposition \ref{prop-semisimple character ratio}, we deduce
        \begin{cor}\label{cor-character ratio and sum on central elements}
        	Keep notations as of Proposition \ref{prop-semisimple character ratio}. If $c=\xi$ with $\xi\in\F_q^\times$ and $\xi^n=1$, then
        	\[\dfrac{I_e(\xi)}{I_e(1)}=\theta^k(\xi^{n/s})\]
        	and
        	\[\mathcal{S}_{(\lambda),(s)}(\xi)=\dfrac{1}{s}\sum_{r\mid\frac{s}{ord(\xi)}}\mu(s/r)(q^r-1)=\left(\frac{1}{s}+o(1)\right)q^{s/ord(\xi)},\]
        	or $o(q^{s/ord(\xi)})$, depending on $\mu(ord(\xi))$.  	
        \end{cor}
        For general types, we need to deal with $B_\rho(h^\rho m:\xi^\rho m')$ more carefully. Now that $\xi^\rho m'$ are all $\xi$, (\ref{equation-B for semisimple}) becomes
        \begin{align}B_\rho(h^\rho m:\xi)=&\sum_{\sigma\in\mathcal{S}_\rho}\prod_{g\in e}\prod_{\tau\in\rho(m,g)}\tau\deg(g)\theta^{\kappa_g\tau}(\xi)\notag\\
        =&\left(\sum_{\sigma\in\mathcal{S}_\rho}\prod_{g\in e}\prod_{\tau\in\rho(m,g)}\tau\deg(g)\right)\left(\prod_{g\in e}\prod_{\tau\in\rho(m,g)}\theta^{\kappa_g\tau}(\xi)\right)\notag\\
        =&z_\rho\prod_{g\in e}\theta^{\kappa_g|\lambda_g|}(\xi)\notag,
        \end{align}
        where $e=(\cdots g^{\lambda_g}\cdots)$. Then for $c=\xi$,  $\bar{\lambda}=(\lambda_1,\cdots,\lambda_k)$ and $\mathbf{s}=(s_1,\cdots,s_k)$, (\ref{equation-semisimple Frobenius partial sum}) becomes
        \begin{align}\mathcal{S}_{\bar{\lambda},\mathbf{s}}(\xi)=&\dfrac{(-1)^{n-\sum_{j=1}^k|\lambda_j|}\phi_{n}(q)}{I_{\bar{\lambda},\mathbf{s}}(1)}\sum_{\rho,m}\dfrac{\beta_\rho(q)\chi(m,e)}{z_{\rho}}\sum_{e\in(\bar{\lambda},\mathbf{s})}z_\rho\prod_{g\in e}\theta^{\kappa_g|\lambda_g|}(\xi)\notag\\
        \label{equation-general Frobenius partial sum on center}=&\dfrac{(-1)^{n-\sum_{j=1}^k|\lambda_j|}\phi_{n}(q)}{I_{\bar{\lambda},\mathbf{s}}(1)}\left(\sum_{e\in(\bar{\lambda},\mathbf{s})}\prod_{g\in e}\theta^{\kappa_g|\lambda_g|}(\xi)\right)\sum_{\rho,m}\beta_\rho(q)\chi(m,e).
        \end{align}
        Recalling Proposition \ref{prop-general character formula}, we get
        \begin{align}
        \sum_{\rho,m}\beta_\rho(q)\chi(m,e)=&\sum_{\rho,m}\prod_{j=1}^k\dfrac{1}{z_{\rho(m,g_j)}}\chi^{\lambda_j}_{\rho(m,g_j)}\beta_{\rho(m,g_j)}(q^{s_j})\notag\\
        =&\prod_{j=1}^k\sum_{|\rho_j|=|\lambda_j|}\dfrac{1}{z_{\rho_j}}\chi^{\lambda_j}_{\rho_j}\beta_{\rho_j}(q^{s_j})\notag\\
        =&\prod_{j=1}^k\{\lambda_j:q^{s_j}\}\notag,
        \end{align}
        Again by Proposition \ref{prop-general character formula},
        \[I_{\bar{\lambda},\mathbf{s}}(1)=(-1)^{n-\sum{|\lambda_j|}}\phi_n(q)\prod_{j=1}^{k}\{\lambda_j:q^{s_j}\},\]
        whence (\ref{equation-general Frobenius partial sum on center}) becomes
        \begin{prop}\label{prop-general Frobenius partial sum on center}
        	With notations as above,
        	\[\mathcal{S}_{\bar{\lambda},\mathbf{s}}(\xi)=\sum_{e\in(\bar{\lambda},\mathbf{s})}\prod_{g\in e}\theta^{\kappa_g|\lambda_g|}(\xi).\]
        \end{prop}
        Together with Lemma \ref{lem-sum over primary roots} we get
        \begin{cor}\label{cor-type of distinct degrees on center}For any distinct type $(\bar{\lambda},\mathbf{s})$,
        	\[\mathcal{S}_{\bar{\lambda},\mathbf{s}}(\xi)=\prod_{j=1}^k\dfrac{1}{s_j}\sum_{r\mid\id_{s_j}(\xi^{|\lambda_j|})}\mu(s_j/r)(q^r-1).\]
        \end{cor}

        Clearly if $ord(\xi)\mid s_j|\lambda_j|$, then $\id_{s_j}(\xi^{|\lambda_j|})=\gcd\left(s_j,\frac{s_j|\lambda_j|}{ord(\xi)}\right)\le\frac{s_j|\lambda_j|}{ord(\xi)}$; otherwise $\id_{s_j}(\xi^{|\lambda_j|})=0$. Hence $\mathcal{S}_{\bar{\lambda},\mathbf{s}}(\xi)\neq 0$ only if $ord(\xi)\mid \gcd(s_1|\lambda_1|,\cdots,s_k|\lambda_k|)$, and the highest power of $q$ in $\mathcal{S}_{\bar{\lambda},\mathbf{s}}(\xi)$ is $q^{n/ord(\xi)}$, which are all achieved by types $(\bar{\lambda},\mathbf{s})$ with $\lambda_j$'s all partitions of $ord(\xi)$. This implies 
        \begin{cor}\label{cor-partial Frobenius sum on center}
        	For any $\xi\in\F_q^\times$ with $ord(\xi)\mid n$,
        	\[\mathcal{S}(\xi)=\left(p(ord(\xi))\sum_{|\rho|=\frac{n}{ord(\xi)}}\prod_{\tau\in\rho}\dfrac{1}{\tau}+o(1)\right)q^{n/ord(\xi)},\]
        	where $p(m)$ is the partition number of $m$.
        \end{cor}
        By model theory or the same argument by Lang-Weil bound used in \cite{Larsen-Lu}, the numerical result above suffices to show that $\dim[,]^{-1}(\xi)=\dim G_n+n/ord(\xi)$. However, we shall avoid the technicality and instead give it direct geometric proof in section \ref{section-geometry of fibers}.
        \medskip

\section{Geometry of fibers over center of $\SL_n(\F_q)$}\label{section-geometry of fibers}
        In general, for any group $G$ and $g\in G$, let $C(g)$ be the centralizer of $g$ in $G$. Then $C(g)$ acts on $[,]^{-1}(g)$ by conjugation via $a\cdot(x,y)=(axa^{-1},aya^{-1})$ for any $a\in C(g)$ and $(x,y)\in[,]^{-1}(g)$, since $[axa^{-1},aya^{-1}]=aga^{-1}=a$. If $g$ belongs to the center of $G$, then $G$ acts on its commutator fiber, which provides a description of the fiber by its $G$-orbits. In our case, the $G_n$-orbits of commutator fibers over the center of $\SL_n$ can be made explicit as follows. Without loss of generality, we work over $\C$.
        
        First, we locate some typical points in the commutator fiber. For $n=ml, l\geq1$, $\xi$ be a primitive $n$-th root of unity and any $\mathbf{b}=(b_1,\cdots,b_{m})\in(\C^{\times})^n$, let 
        \begin{equation}\label{equation-definition of tau}
        \tau_n^{\mathbf{b}}(\xi)=\diag(b_1,\cdots,b_{m},\xi^mb_1,\cdots,\xi^mb_{m},\cdots,\xi^{(l-1)m}b_1,\cdots,\xi^{(l-1)m}b_{m})
        \end{equation}
        and $\sigma\in\GL_n(\C)$ be the $n$-cycle $(1\ 2\ \cdots\ n)$. Then we can simply compute and verify that 
        \[[\sigma^m,\tau_n^{\mathbf{b}}(\xi)]=\xi^m.\]
        To get more sophisticated points in the fiber, we define a subgroup of $\GL_n(\C)$ as follows.
        
        Suppose $m=m_1+\cdots+m_k$ with $m_i\geq 1$ and $b_1=\cdots=b_{m_1}=\beta_1$, $b_{m_1+1}=\cdots=b_{m_1+m_2}=\beta_2$, $\cdots, b_{m_1+\cdots+m_{k-1}+1}=\cdots=b_m=\beta_k$ with $\beta_i\beta_j^{-1}\neq\xi^{rm}$ for $i\neq j$ and any integer $r$, then 
        \[C(\tau_n^{\mathbf{b}}(\xi))\simeq(\GL_{m_1}(\C)\times\cdots\times\GL_{m_k}(\C))^l,\]
        with each $\GL_{m_i}(\C)$ arranged diagonally in correspondence with appearance of $\beta_i$. 
        Thus for any $A,B\in C(\tau_n^{\mathbf{b}}(\xi))$, they are of the form 
        \begin{equation}\label{equation-shape of A}A\sim \diag(A_{11},\cdots,A_{1k};\cdots;A_{l1},\cdots,A_{lk}),\ A_{ij}\in\GL_{m_j}(\C), \forall 1\leq i\leq l, 1\leq j\leq k,\end{equation}
        \[B\sim \diag(B_{11},\cdots,B_{1k};\cdots;B_{l1},\cdots,B_{lk}),\ B_{ij}\in\GL_{m_j}(\C), \forall 1\leq i\leq l, 1\leq j\leq k,\]
        and if $A=[\sigma^{-m},B]$
        \begin{align}&A=\sigma^{-m}B\sigma^mB^{-1}\notag\\
        =&\diag(B_{l1}B_{11}^{-1},\cdots,B_{lk}B_{1k}^{-1};B_{11}B_{21}^{-1},\cdots,B_{2,k}^{-1};\cdots;B_{(l-1)1}B_{l1}^{-1},\cdots,B_{(l-1)k}B_{lk}^{-1})\notag\\
        \label{equation-defining condition 5.2}\Rightarrow&A_{1j}A_{2j}\cdots A_{lj}=B_{lj}B_{1j}^{-1}B_{1j}B_{2j}^{-1}\cdots B_{(l-1)j}B_{lj}^{-1}=I_n, \ \forall 1\leq j\leq k.
        \end{align}
        We define by $\SL_n^{\bar{m}_{\mathbf{b}}}(\C)$ the subgroup consisting of $A\in\GL_n(\C)$ in the form of (\ref{equation-shape of A}) which satisfies (\ref{equation-defining condition 5.2}).
        
        Then we prove
        \begin{thm}\label{thm-geometry of fiber over center}
        	For $n=ml, l\geq1$, let $\xi$ be a primitive $n$-th root of unity and $D$ the set of all semi-simple elements in $\GL_n(\C)$,  then we have
        	\[[,]^{-1}(\xi^m)\cap (\GL_n(\C)\times D)=\bigsqcup_{\mathbf{b}\in(\C^{\times})^m,A\in C(\tau_n^{\mathbf{b}}(\xi))/C_{\mathbf{b}}}\GL_n(\C)\cdot(\sigma^mA,\tau_n^{\mathbf{b}}(\xi)),\]
        	in which $C_{\mathbf{b}}=C(\tau_n^{\mathbf{b}}(\xi))\cap\SL_n^{\bar{m}_{\mathbf{b}}}(\C)$ with $C(\tau_n^{\mathbf{b}}(\xi))$ the centralizer of $\tau_n^{\mathbf{b}}(\xi)$ and $\SL_n^{\bar{m}_{\mathbf{b}}}(\C)$ defined by (\ref{equation-defining condition 5.2}).\\       	
        	Consequently we have 
        	\[\dim[,]^{-1}(\xi^m)\cap (\GL_n(\C)\times D)=\dim\GL_n(\C)+m.\]
        \end{thm}
        \begin{proof}        
        	From the equation $xyx^{-1}=\xi^my$ with $y$ semi-simple, we can see that the action of $\xi^m$ by multiplication on the spectrum of $y$ is partitioned into $m$ many $l-$cycles, hence we can make $y=\tau_n^{\mathbf{b}}(\xi)$ for some $\mathbf{b}\in(\C^{\times})^m$  by conjugation. Then compared with $\sigma^{m} \tau_n^{\mathbf{b}}(\xi)\sigma^{-m}=\xi^m \tau_n^{\mathbf{b}}(\xi)$ we get \[(\sigma^{-m}x)\tau_n^{\mathbf{b}}(\xi)=\tau_n^{\mathbf{b}}(\xi)(\sigma^{-m}x),\]
        	i.e. $\sigma^{-m}x\in C(\tau_n^{\mathbf{b}}(\xi))$. 
        	We need to know whether $(x,\tau_n^{\mathbf{b}}(\xi))$ belongs to the same orbit with $(\sigma^m,\tau_n^{\mathbf{b}}(\xi))$ under the conjugation action of $\GL_n(\C)$. Knowing that $y$ can always be conjugated to $\tau_n^{\mathbf{b}}(\xi)$, we just need to consider the action of $C(\tau_n^{\mathbf{b}}(\xi))$.\\
        	First, we have 
        	$[C(\tau_n^{\mathbf{b}}(\xi))\sigma^{-m},C(\tau_n^{\mathbf{b}}(\xi))]\subset C(\tau_n^{\mathbf{b}}(\xi))\cap \SL_n(\C)$, or in other words $C(\tau_n^{\mathbf{b}}(\xi))$ acts on $\sigma^mC(\tau_n^{\mathbf{b}}(\xi))$ via conjugation: $\forall A,B\in C(\tau_n^{\mathbf{b}}(\xi))$,
        	\begin{align}
        	&(\sigma^m A)^{-1}B(\sigma^m A)B^{-1}\tau_n^{\mathbf{b}}(\xi)\notag\\
        	=&(\sigma^m A)^{-1}B\sigma^m\tau_n^{\mathbf{b}}(\xi)AB^{-1}\notag\\
        	=&(\sigma^m A)^{-1}B(\xi^m\tau_n^{\mathbf{b}}(\xi)\sigma^m) AB^{-1}\notag\\
        	=&\xi^m A^{-1}\sigma^{-m}\tau_n^{\mathbf{b}}(\xi)B\sigma^m AB^{-1}\notag\\
        	=&\xi^mA^{-1}(\xi^{-m}\tau_n^{\mathbf{b}}(\xi)\sigma^{-m})B\sigma^m AB^{-1}\notag\\
        	=&\tau_n^{\mathbf{b}}(\xi)(\sigma^m A)^{-1}B(\sigma^m A)B^{-1}\notag\\
        	\Rightarrow&(\sigma^m A)^{-1}B(\sigma^m A)B^{-1}\in C(\tau_n^{\mathbf{b}}(\xi))\notag\\
        	\Rightarrow&B(\sigma^m A)B^{-1}\in\sigma^m AC(\tau_n^{\mathbf{b}}(\xi))=\sigma^m C(\tau_n^{\mathbf{b}}(\xi)).
        	\end{align}
        	Moreover, since
        	\[B(\sigma^m A)B^{-1}=\sigma^m A'\Leftrightarrow [(\sigma^m A)^{-1},B]=A^{-1}A',\]
        	$\sigma^mA$ and $\sigma^mA'$ belong to the same orbit if and only if their difference $A^{-1}A'\in[(\sigma^m A)^{-1},C(\tau_n^{\mathbf{b}}(\xi))]$. We claim that $[\sigma^{m}, C(\tau_n^{\mathbf{b}}(\xi))]=C(\tau_n^{\mathbf{b}}(\xi))\cap\SL_n^{\bar{m}_{\mathbf{b}}}(\C)$.
        	
        	Directly $[\sigma^{m}, C(\tau_n^{\mathbf{b}}(\xi))]\subset C(\tau_n^{\mathbf{b}}(\xi))\cap\SL_n^{\bar{m}_{\mathbf{b}}}(\C)$ follows from definition. Conversely, if $A$ is of the form (\ref{equation-shape of A}) and satisfies the condition (\ref{equation-defining condition 5.2}), then $\forall B_j\in\GL_{m_j}(\C), 1\leq j\leq k$, setting $B_{ij}=A_{ij}^{-1}\cdots A_{1j}^{-1}B_j, \forall 1\leq i\leq k$, we have $A=\sigma^{-m}B\sigma^mB^{-1}$, whence proof of the claim.
        	
        	Altogether we know that if $(x,y)\in[,]^{-1}(\xi^m)$ with $y$ semi-simple then it falls in some orbit of the form $\GL_n(\C)\cdot(\sigma^mA,\tau_n^{\mathbf{b}}(\xi))$, for some $A\in C_{\mathbf{b}}:=C(\tau_n^{\mathbf{b}}(\xi))\cap\SL_n^{\bar{m}_{\mathbf{b}}}(\C)$. Note that $\SL_n^{\bar{m}_{\mathbf{b}}}(\C)$ is the kernel of the homomorphism from $C(\tau_n^{\mathbf{b}}(\xi))$ to $\GL_{m_1}(\C)\times\cdots\times\GL_{m_k}(\C)$ defined by the left hand side of (\ref{equation-defining condition 5.2}),  $C_{\mathbf{b}}$ has a one-to-one correspondence with $\GL_{m_1}(\C)\times\cdots\times\GL_{m_k}(\C)$.
        	
        	Now we count the dimension of the fiber over $\xi^m$. By (\ref{equation-defining condition 5.2}) we see that any matrix in the stabilizer of $(\sigma^{m},\tau_n^{\mathbf{b}}(\xi))$ must have the form
        	\[A\sim \diag(A_{11},\cdots,A_{1k};\cdots;A_{11},\cdots,A_{1k}),\]
        	i.e. $A_{ij}$ are all the same for $j$ fixed. Hence the dimension of any such stabilizer is $\sum_{i=1}^{k}\dim\GL_{m_i}(\C)$, and
        	\[\dim\GL_n(\C)\cdot(\sigma^{m}, \tau_n^{\mathbf{b}}(\xi))=\dim\GL_n(\C)-\sum_{i=1}^{k}\dim\GL_{m_i}(\C),\]
        	Also by solving the linear equations given by $B\sigma^mAB^{-1}=\sigma^mA$ in $(\GL_{m_1}(\C)\times\cdots\times \GL_{m_k}(\C))^l\simeq C(\tau_n^{\mathbf{b}}(\xi))$, we can see that stabilizers of $(\sigma^{m}A,\tau_n^{\mathbf{b}}(\xi))$ all have the same dimension as above. Consequently the dimension of the union of all orbits of such type (determined by $\mathbf{b}$ as described in the second paragraph of this section) is
        	\[\sum_{i=1}^{k}\dim\GL_{m_i}(\C)+k+\dim\GL_n(\C)-\sum_{i=1}^{k}\dim\GL_{m_i}(\C)=k+\dim\GL_n(\C),\]
        	which achieves maximum when $k=m$.    	
        \end{proof}
        For $[x,y]=xyx^{-1}y^{-1}=\xi^m$ in $\GL_n(\C)$ with $y$ not necessarily semisimple, we can use the multiplicative Jordan decomposition to reformulate the commutator equation as follows
        \begin{align*}
        &xyx^{-1}=\xi^my\Leftrightarrow xy_sx^{-1}xy_ux^{-1}=\xi^my_sy_u\\
        \Leftrightarrow &xy_sx^{-1}=\xi^my_s \text{ and }  xy_ux^{-1}=y_u\\
        \Leftrightarrow & [x,y_s]=\xi^m\text{ and } [x,y_u]=1.
        \end{align*}
        Then by Theorem \ref{thm-geometry of fiber over center} which deals with the first equation, and further confined by the second equation, we have 
        \begin{thm}
        	With the notations and proof above,
        	\[[,]^{-1}(\xi^m)=\bigcup_{|\lambda_1|+\cdots+|\lambda_k|=m}\bigcup_{\mathbf{b}\in(\C^\times)^k, A\in\GL_n(\C)/C_{\mathbf{b}}}\GL_n(\C)\cdot(\sigma^mA,\tau_{\lambda_1,\cdots,\lambda_k}^{\mathbf{b}}(\xi)),\]
        	in which for $\mathbf{b}=(b_1,\cdots,b_k)$, \[\tau_{\lambda_1,\cdots,\lambda_k}^{\mathbf{b}}(\xi)=\diag(U_{\lambda_1}(t-b_1),\cdots,U_{\lambda_k}(t-b_k);U_{\lambda_1}(t-b_1\xi^m),\cdots,U_{\lambda_k}(t-b_k\xi^m);
        	\]\[\cdots;U_{\lambda_1}(t-b_1\xi^{(l-1)m}),\cdots,U_{\lambda_k}(t-b_k\xi^{(l-1)m})),\] 
        	for any partitions $|\lambda_i|=m_i$, $m=m_1+m_2+\cdots+m_k$ with $m_i\geq 1$ and $b_1,\cdots,b_k$ have distinct $l$-cycles, i.e. $b_ib_j^{-1}\neq\xi^{rm}, \forall i\neq j, 0\leq r\leq l-1$. (Note that if $k=m$, then $\lambda_i=\{1\}$ and $\tau_{\lambda_1,\cdots,\lambda_k}^{\mathbf{b}}(\xi)=\tau_{n}^{\mathbf{b}}(\xi)$.)\\
        	Consequently we have 
        	\[\dim[,]^{-1}(\xi^m)=\dim\GL_n(\C)+m.\]
        \end{thm}


\medskip
        \begin{thebibliography}{9}
        \bibitem{Frob}
        F.G. Frobenius, \textit{\"{U}ber Gruppencharaktere}, Sitzber. Preuss. Akad. Wiss (1896) 985-1021; reprinted in Gesammelte Abhandlungen, Vol. 3 (Springer, Heidelberg, 1968) 1-37.
        \bibitem{Green}
        J. A. Green, \textit{The characters of the finite general linear group}, Trans. Amer. Math. Soc. 80 (1955), 402-447.
        \bibitem{Hall}
        Philip Hall, \textit{On Representatives of Subsets}, J. London Math. Soc., 10 (1): 26–30 (1935).
        \bibitem{Larsen-Lu}
        Michael Larsen, Zhipeng Lu, \textit{Flatness of commutator map over $\SL_n$}, arXiv:1807.07300
        \bibitem{Littlewood}
        D.E. Littlewood, \textit{The theory of group characters and matrix representations of groups}, Oxford, 1940.
        \bibitem{Macdonald}
        I. G. Macdonald, \textit{Symmetric functions and Hall polynomials}. Second edition. Oxford Mathematical Monographs. Oxford Science Publications. The Clarendon Press, Oxford University Press, New York, 1995.
        
        
        \end{thebibliography}
\end{document}